\documentclass{amsart}

\usepackage[vmargin=3cm, hmargin=3cm]{geometry}
\usepackage{amsmath,amssymb,amsfonts,epsf} 
\usepackage{xypic}

\newtheorem{Thm}{Theorem}[section]
\newtheorem{prop}[Thm]{Proposition} 
\newtheorem{lem}[Thm]{Lemma} 
 
\newtheorem{defi}[Thm]{Definition}

\newcommand{\sd}{\rtimes}
\newcommand{\R}{\mathbb R}

\title[Formality of the 2-discs and semidirect products]{Formality of the framed little 2-discs operad and semidirect products}

\date{23 November, 2009}

\author{Jeffrey Giansiracusa}
\address{
Mathematical Institute, University of Oxford \\ 
24 - 29 St. Giles\\
Oxford \\
OX1 3LB\\ 
United Kingdom}
\email{giansira@maths.ox.ac.uk}

\author{Paolo Salvatore} 
\address{
Dipartimento di Matematica \\
Universita' di Roma ``Tor Vergata'' \\
Via della Ricerca Scientifica \\
00133 Roma  \\
ITALY}
\email{salvator@mat.uniroma2.it}


\begin{document}

\begin{abstract}
  We prove that the operad of framed little 2-discs is formal.
  Tamarkin and Kontsevich each proved that the unframed 2-discs operad
  is formal.  The unframed 2-discs is an operad in the category of
  $S^1$-spaces, and the framed 2-discs operad can be constructed from
  the unframed 2-discs by forming the operadic semidirect product with
  the circle group.  The idea of our proof is to show that
  Kontsevich's chain of quasi-isomorphisms is compatible with the
  circle actions and so one can essentially take the operadic
  semidirect product with the homology of $S^1$ everywhere to obtain a
  chain of quasi-isomorphisms between the homology and the chains of
  the framed 2-discs.
\end{abstract}

\maketitle

\section{Introduction}

We begin by recalling two closely related operads.  First, let $D_2$
denote the \emph{little 2-discs operad} of Boardman and Vogt.  In
arity $n$ it is the space of embeddings of the union of $n$ discs into
a standard disc, where each disc is embedded by a map which is a
translation composed with a dilation.  At the level of spaces, group
complete algebras over this operad are 2-fold loop spaces, and at the
level of homology an algebra over $H_*(D_2)$ is precisely a
Gerstenhaber algebra.

A variant of the $D_2$ operad is the \emph{framed little 2-discs}
operad, denoted $fD_2$, introduced by Getzler \cite{Getzler}.  Here
the little discs are allowed to be embedded by a composition of a
dilation, rotation, and translation.  The (unframed) little 2-discs
operad $D_2$ is an operad in the category of $S^1$-spaces, where the
circle acts by conjugation.  Markl and Salvatore-Wahl \cite{SW}
presented the framed little 2-discs operad as the semidirect product
of the circle group $S^1$ with $D_2$.  In particular, $fD_2(n)=D_2(n)
\times (S^1)^n$.  Getzler observed that algebras over the homology
operad $H_*(fD_2)$ are precisely Batalin-Vilkovisky algebras, and at
the space level Salvatore-Wahl proved that a group complete algebra
over $fD_2$ is a 2-fold loop space on a based space with a circle
action.

The operad $D_2$ is homotopy equivalent to the Fulton-MacPherson
operad $FM=FM_2$ \cite{Barcelona} (we drop the subscript since we will
only be discussing 2-discs in this note); the space $FM(n)$ is a
compactification of the configuration space of $n$ ordered distinct
points in the plane modulo translations and positive dilations.  As
with $D_2$, the circle acts on $FM$ by rotations.  The semidirect
product construction for this action gives the \emph{framed
  Fulton-MacPherson operad} $fFM = fFM_2$, which is homotopy
equivalent to $fD_2$, and such that $fFM(n)=FM(n) \times (S^1)^n$.
Both $FM$ and $fFM$ are operads of semi-algebraic sets.  

Tamarkin \cite{Tamarkin} and Kontsevich \cite{Ko1, Ko} proved the
following formality theorem.
\begin{Thm}\label{unframed-formality}
  The operad $C_*(FM)$ of chains on $FM$ with real coefficients is
  quasi-isomorphic to its homology operad $H_*(FM)$, the Gerstenhaber
  operad.
\end{Thm}
(One can also use singular or semi-algebraic chains in the statement;
we will return to this point later.)  Kontsevich's proof seems more
geometric and has the advantage of extending to a proof of formality
for the little $k$-discs for all $k\geq 2$; this proof has been
explained in greater detail by Lambrechts and Volic \cite{LV}.  

In general, formality of an operad is a powerful property with many
theoretical and computational applications.  The above operad
formality theorem plays an important role in Tamarkin's proof
\cite{Tamarkin-deformation} of Kontsevich's deformation quantization
theorem.  Our purpose in writing this note is to show that
Kontsevich's proof of formality of the operad $FM$ can be adapted to
show the formality of the operad $fFM$.  Our main result is:
\begin{Thm}\label{main}
  The operad $C_*(fFM)$ of chains on $fFM$ with real coefficents is
  quasi-isomorphic to its homology operad $H_*(fFM)=BV$, the
  Batalin-Vilkovisky operad.
\end{Thm} 
An independent proof of this formality, built from Tamarkin's method
rather than Kontsevich's, is due to Severa \cite{Severa}.  One
interesting application of this formality result is given in
\cite{Vallette}, where it is used to construct homotopy BV algebra
structures on objects such as the chains on double loop spaces.

Unline $fD_2$, the operad $fFM$ in fact has the structure of a
\emph{cyclic operad}, and so it is natural to ask if the formality can
be made compatible with the cyclic structure. After the present work
was completed we found a proof \cite{cyclic-paper} of the stronger
result that $fFM$ is formal as a cyclic operad, although that proof is
significantly more involved and required the introduction of a new
type of graph complex.

Kontsevich's proof showed formality of the little $k$-discs operad for
\emph{all} $k$, and so it is reasonable to ask if the framed $k$-discs
operads are all formal as well (as operads, or better yet as cyclic
operads).  We plan to address this question in future work.  The
proofs given in this paper, \cite{Severa}, and \cite{cyclic-paper}
address only the case $k=2$.  These arguments do not work for $k>2$
for various reasons.  In the Tamarkin formality argument it is
essential that the operad spaces are $K(\pi,1)$s, and this is no
longer true for $k>2$.  The argument in this paper does not
immediately extend to higher $k$ becasue one would have to replace the
group $S^1=SO(2)$ with $SO(k)$ and find a quasi-isomorphism  $H^*SO(k) \to \Omega^*SO(k)$ that is compatible
with Kontsevich's integration map; we do not know if this is possible.
There are similar obstacles to adapting the argument in
\cite{cyclic-paper} to higher $k$.

\subsection{Outline of the proof}

First recall the outline of Kontsevich's proof of Theorem
\ref{unframed-formality}.  It goes by constructing a certain
DG-algebra $G(n)$ of graphs together with a quasi-isomorphism $I:G(n)
\to \Omega^*(FM(n))$ to the DG-algebra of semi-algebraic forms, and a
projection $G(n) \to H^*(G(n))=H^*(FM(n))$ that is also a
quasi-isomorphism.  Both of these quasi-isomorphisms are essentially
morphisms of DGA cooperads (this is not quite true --- the subtleties
here are discussed nicely in \cite{LV}).  By dualizing one can
obtain from this a chain of quasi-isomorphisms giving formality of
$FM$.  

What we show
in this note is that Kontsevich's formality proof is in compatible in
a precise sense with the circle action.  
The circle action on $FM$ makes $H^*(FM)$ into a cooperad in $H:=
H^*(S^1)$-comodules.  Kontsevich's DGA cooperad of admissible graphs
$G$ has a differential given by contracting edges; we define a degree
$-1$ derivation $\Delta$ on $G(n)$ given by deleting edges.  This
derivation defines a $H$-comodule structure on $G(n)$; we check that
that this comodule structure is compatible with the cooperad structure
and that the projection $G(n) \to H^*(FM(n))$ is a morphism of
$H$-comodules.  Using the quasi-isomorphism $H^*(S^1) \to
\Omega^*(S^1)$, this morally allows us to form a diagram of
quasi-isomorphisms of semidirect product cooperads
\[
\Omega^*(FM\sd S^1) \leftarrow G \sd H \to H^*(FM) \sd H \cong
H^*(FM\sd S^1).
\]
However, the proof is not quite so simple because the functor of
semi-algebraic forms is contravariant monoidal and so $\Omega^*(FM\sd
S^1)$ is not a cooperad on the nose.  Nevertheless, this issue can be
overcome easily, exactly as discussed in \cite{LV} in the unframed case.

\section{A degree $-1$ derivation on admissible graphs}

Consider the circle action $\rho_n:S^1 \times FM(n) \to FM(n)$, and
let $H$ denote the coalgebra $H^*(S^1) = \R[d\theta]$.  The circle
action induces an $H$-comodule structure on $H^*(FM(n))$; the coaction
is given by the formula
\[
\rho_n^*(x)=[d\theta] \otimes \Delta(x) + 1 \otimes x
\]
where
\[
\Delta:H^*(FM(n)) \to H^{*-1}(FM(n))
\] 
is a degree -1 derivation.  Clearly the $H$-comodule structure and
the derivation $\Delta$ determine each other.

We shall now lift the $H$-comodule structure to Kontsevich's
admissible graph complex $G(n)$ by lifting the derivation $\Delta$.
Recall that $G(n)$ is the complex of real vector spaces spanned by
admissible graphs on $n$ external vertices \cite{Ko}.  The grading is
defined by
\[
(\# \mbox{ edges}) - 2(\# \mbox{ internal vertices}).
\]
Graphs are \emph{admissible} if they are at least trivalent at each
internal vertex and satisfy a few additional conditions.  Each graph
is equipped with a total ordering of its edges, and a permutation of
the edges acts on the corresponding generator of $G(n)$ by its sign.
Given a graph $g$ with edges $e_1,\dots,e_k$, the differential is
defined by
\[
dg =\sum_i (-1)^i g/e_i
\]
where $g/e_i$ is the graph obtained from $g$ by collapsing the edge $e_i$.
Any non-admissible terms occuring in the sum are set to zero. 
Recall that the complex $G(n)$ has a graded commutative algebra
structure given by disjoint union of internal vertices and the union
of edges.  The Kontsevich integral defines a morphism of differential
graded algebras $I:G(n) \to \Omega^*(FM(n))$ (the target is the
algebra of semi-algebraic forms defined in \cite{HLTV}) and this is a
quasi-isomorphism.

\begin{defi}
  A linear operator $\Delta:G(n) \to G(n)$ of degree -1 is defined as
  follows.  Given a graph $g \in G(n)$ with ordered set of edges
  $e_1,\dots,e_k$,
  \[
  \Delta(g)=\sum_i (-1)^{i+1}(g-e_i)
  \]
  where $g-e_i$ is the graph obtained by deleting the edge $e_i$ from
  $g$ (without identifying the endpoints together). If a summand is a
  non-admissible graph then we set it to zero.
\end{defi}

\begin{prop}
The operator $\Delta$ satisfies:
\begin{enumerate}
\item $\Delta^2=0$;
\item it is a derivation of the algebra $G(n)$;
\item it graded commutes with the differential $d$ of $G(n)$, i.e. 
$d\Delta = -\Delta d$.
\end{enumerate}
Hence the rule
\[
g \mapsto 1 \otimes g + [d\theta] \otimes \Delta(g)
\]
gives $G(n)$ the structure of a DG-comodule over
the coalgebra $H:=H^*(S^1)$.
\end{prop}

Let $\theta_{ij}:FM(n) \to S^1$ be the map measuring the angle of the
line from the $i$-th to the $j$-th point with the first coordinate
axis.  The algebra $G(n)$ is freely generated by {\em indecomposable}
graphs, those that do not get disconnected by removing a small
neighbourhood of the set of external vertices.  If $g$ is
indecomposable with no internal vertices then it has only an edge
between some vertices $i$ and $j$, and we denote it $g=\alpha_{ij}$.
Kontsevich considers the algebra map 
\[
q_n:G(n) \to H^*(G(n))=H^*(FM(n))
\]
sending all graphs with internal vertices to 0, and such that
\[
q_n(\alpha_{ij})=\theta_{ij}^*(d\theta):=d\theta_{ij}.
\]

\begin{prop}[Kontsevich, Lambrechts-Volic]\label{quasi}
  The collection of maps $\{q_n\}$ assemble to a quasi-isomorphism of
  DG-cooperads $q:G \to H^*(FM)$.
\end{prop}

\begin{prop} \label{q}
The projection $q_n:G(n) \to H^*(FM(n))$ is a map of $H$-comodules, i.e.
$q\circ \Delta = \Delta \circ  q $.
\end{prop}

\begin{proof}
  For any graph $g$, the summands of $\Delta(g)$ have the same number
  of internal vertices as $g$. Therefore if $g$ is indecomposable with
  some internal vertices then $q(\Delta(g))=\Delta(q(g))=0$.  If $g$
  is indecomposable with no internal vertices, then $g=\alpha_{ij}$
  for some $i,j$, and $\Delta(g)=1$ (the unit of the algebra $G(n)$ is
  the graph on $n$ external vertices with no edges). Then
  $q(g)=[d\theta_{ij}]$. Since the map $\theta_{ij}$ is
  $S^1$-equivariant, we have that
  \[
  \Delta([d\theta_{ij}])=\theta_{ij}^*(\Delta([d\theta]))=\theta_{ij}^*(1)=1
  \]
  and so $\Delta(q(g))=1=q(\Delta(g))$.
\end{proof}

\section{Compatibility of $\Delta$ with the Kontsevich integral and
  the cooperad structures}

We show in the next lemma that the operator $\Delta$ is compatible
with the integration map $I:G(n) \to \Omega^*(FM(n))$.
\begin{lem} \label{formula}
For $g \in G(n)$, 
\[
\rho_n^*(I(g))=d\theta \times I(\Delta(g)) + 1 \times I(g) \in
\Omega^*(S^1 \times FM(n)).
\]
\end{lem}

\begin{proof} 
  If $g=\alpha_{ij}$, then $I(g)=d\theta_{ij}$, $\Delta(g)=1$, and
  $\rho_n^*(d\theta_{ij})=d\theta \times 1 + 1 \times d\theta_{ij}$.
  If $g$ is a graph with no internal vertices, then it is a product of
  forms $d\theta_{ij}$ and the result follows by multiplying the
  corresponding expressions, since $\rho_n^*, I$ are algebra maps and
  $\Delta$ is a derivation.  If $g$ has $k$ internal vertices then
  $I(g)=p_*(I(h))$, for some $h \in G(n+k)$, where $p_*$ denotes the
  push-forward along the semi-algebraic bundle projection $p:F(n+k)
  \to F(n)$. It follows from the definition of $I$ that
  $p_*(I(\Delta(h)))=I(\Delta(g))$. The diagram
  \[
  \xymatrix{
    S^1 \times F(n+k)\ar[r]^{\rho_{n+k}}\ar[d]_{S^1 \times p}& F(n+k)\ar[d]^{p}\\
    S^1 \times F(n)\ar[r]^{\rho_n}& F(n) }
  \]
  is a pullback of semi-algebraic sets. By Proposition 8.13 in
  \cite{HLTV}
  \[
  (\rho_n)^* \circ p_* = (S^1 \times p)_* \circ \rho_{n+k}^* .
  \]
  Since $h$ has no internal vertices
  \[
  \rho_{n+k}^*(I(h))=d\theta \times I(\Delta(h)) + 1 \times I(h)
  \]
  and so
  \begin{align*}
    \rho_n^*(I(g)) & = \rho_n^*(p_*(I(h))) = 
                       (S^1 \times p)_* (\rho_{n+k}^*(I(h))) \\
    & = d\theta \times p_*(I(\Delta(h))) + 1 \times p_*(I(h)) \\
    & = d\theta \times I(\Delta(g)) + 1 \times I(g).  \qedhere
  \end{align*}
\end{proof}

We show next that the operator $\Delta$ is compatible with the
cooperad structure of $G$ constructed by Kontsevich.  The tensor
product of $H$-comodules is a $H$-comodule, such that $\Delta$ on the
tensor product is defined by the Leibniz rule.

\begin{prop}\label{H-comodule-cooperad}
  The cooperad structure map 
\[
\circ_i : G(m+n-1) \to G(m) \otimes G(n)
\]
(with $1 \leq i \leq m$) commutes with $\Delta$; i.e. it is a map of
$H$-comodules.
\end{prop}

\begin{proof} 
  Given a graph $g \in G(m+n-1)$, 
  \[
  \circ_i(g)=\sum_j (-1)^{s(j)} g'_{j} \otimes g''_j,
  \]
  where $j$ ranges over partitions of the set $V$ of internal vertices
  into two sets $V'_j$ and $V''_j$. Then $g''_j \in G(n)$ is the full
  subgraph of $g$ containing the external vertices $\{i,\dots,i+n-1\}$
  (relabelled) and the internal vertices in $V''_j$ .  The graph
  $g'_j$ is obtained from $g$ by collapsing $g''_j$ to a single
  external vertex, and relabelling external vertices.  The sign $s(j)$
  is the sign of the permutation moving the edges from the ordering of
  $g$ to the ordering of $g'_j$ followed by the ordering of $g''_j$.
  If such graphs have repeated edges or are not admissible then they
  are identified to zero. From the definition one sees that
  \[
  \circ_i(\Delta(g))=\sum_j (-1)^{s(j)} (\Delta(g'_j) \otimes g''_j +
  (-1)^{|g'_j|}g'_j \otimes \Delta(g''_j)) = \Delta (\circ_i(g)). \qedhere
  \]
\end{proof}

\section{From $H$-comodules to semidirect product cooperads}

By Proposition \ref{H-comodule-cooperad} above one can form the
semidirect product cooperad $G \sd H$ with $(G \sd H)(n)=G(n) \otimes
H^{\otimes n}$, by extending to the differential graded setting the
construction in section 4 of \cite{SW}, and dualizing it.  Explicitly
the cooperad structure maps
\[
\circ_i :(G \sd H)(m+n-1) \to (G \sd H)(m) \otimes (G \sd H)(n)
\]
are the algebra maps defined by sending
\begin{align*}
d\theta_k & \mapsto d\theta'_i + d\theta''_{k-i+1} 
   \mbox{ for $i \leq k \leq n+i-1$}, \\
d\theta_k & \mapsto d\theta'_k \mbox{ for $k<i$}, \\
d\theta_k & \mapsto d\theta'_{k-i+1} \mbox{ for $k \geq n+i$},
\end{align*}
and for $g \in G(m+n-1)$,
\[
g \mapsto   \sum_j (-1)^{s(j)} ( g'_j  \otimes d\theta'_i \otimes
\Delta(g''_j) +   g'_j  \otimes g''_j).
\]

By Proposition \ref{q} and Proposition \ref{quasi} the collection $q$
induces a quasi-isomorphism of cooperads
\[
G \sd H \to H^*(FM) \sd H = H^*(fFM).
\]

The semi-algebraic differential forms on $FM$ do not exactly
constitute a cooperad because semi-algebraic forms is a
contravariant monoidal functor.  The cross product of forms
$\Omega^*(FM(m)) \otimes \Omega^*(FM(n)) \to \Omega^*(FM(m)\times
FM(n))$ (which is a quasi-isomorphism) and the operad
composition $\circ_i: FM(m) \times FM(n) \to FM(m+n-1)$ induce a zigzag
\[
\Omega^*(FM(m+n-1)) \rightarrow \Omega^*(FM(m)\times FM(n))
\leftarrow \Omega^*(FM(m)) \otimes \Omega^*(FM(n)).
\]
Nevertheless, there is a compatibility rule between operadic
composition in $G$ and $FM$.
\begin{lem}[Lemma 8.19 of \cite{LV}]\label{combo}
The pullback along the operad composition map 
$\circ_i^{FM}:FM(m) \times FM(n) \to FM(m+n-1)$, for $g \in G(m+n-1)$,  gives
\[
(\circ_i^{FM})^*(I(g))= \sum_j (-1)^{s(j)} I(g'_j) \times I(g''_j)
\]
where $\circ_i^G(g) = \sum_j (-1)^{s(j)}g'_j \otimes g''_j$.
\end{lem}

We state next an analogous compatibility condition for the {\em
  framed} case.  There are quasi-isomorphisms
\[
\Omega^*(FM(n)) \otimes H^{\otimes n} \to \Omega^*(FM(n)) \otimes
\Omega^*(S^1)^{\otimes n} \to \Omega^*(FM(n)\times (S^1)^n)=
\Omega^*(fFM(n)).
\]
The first map sends fundamental classes of circles to volume forms,
and the second map is the cross product of forms. 
The composition with the Kontsevich integral gives a quasi-isomorphism of algebras
$$I':G(n) \otimes H^{\otimes n}=(G \sd H)(n) \to  \Omega^*(fFM(n)).$$

\begin{lem}\label{main-diagram-lemma}
The diagram
\[
\xymatrix{
  (G \sd H)(m+n-1)  \ar[rr]^{\circ^{G\sd H}_i} \ar[d]^{I'} & &  
  (G \sd H)(m)  \otimes (G \sd H)(n) \ar[d]^{I' \otimes I'} \\
  \Omega^*(fFM(m+n-1)) \ar[r]^{(\circ^{fFM}_i)^*} & 
  \Omega^*(fFM(m) \times fFM(n)) &
  \Omega^*(fFM(m)) \otimes \Omega^*(fFM(n)) \ar[l]
}
\]
 commutes.
\end{lem}

\begin{proof} 
By definition of semidirect product the composition in $fFM$ is
$$(x,z_1,\dots,z_m) \circ_i^{fFM} (y,w_1,\dots,w_n)=
(x \circ^{FM}_i \rho_m(z_i,y),z_1,\dots,z_{i-1},z_i w_1,\dots,
z_i w_n,z_{i+1},\dots,z_m).$$
The lemma follows from this, Lemma 
\ref{formula} and Lemma \ref{combo}.
\end{proof}

We proceed similarly as in section 10 of \cite{LV} observing that
integration on semi-algebraic chains of forms associated to graphs
defines pairings
\[
C_*(fFM(n)) \otimes (G \sd H)(n) \to \R
\]
sending $c \otimes g \mapsto \int_c I(g)$, and their adjoints give a
quasi-isomorphism of operads
\[
C_*(fFM) \to (G \sd H)^* .
\]
This together with the fact that $q^*:H_*(fFM) \to (G \sd H)^*$ is a
quasi-isomorphism of operads establishes theorem \ref{main}.

\end{document}